\documentclass{amsart}
\usepackage[T1]{fontenc}
\usepackage{lmodern}
\usepackage{amssymb,amsmath}
\usepackage{ifxetex,ifluatex}
\usepackage{fixltx2e} 
\IfFileExists{upquote.sty}{\usepackage{upquote}}{}
\ifnum 0\ifxetex 1\fi\ifluatex 1\fi=0 
  \usepackage[utf8]{inputenc}
\else 
  \ifxetex
    \usepackage{mathspec}
    \usepackage{xltxtra,xunicode}
  \else
    \usepackage{fontspec}
  \fi
  \defaultfontfeatures{Mapping=tex-text,Scale=MatchLowercase}
  
\fi
\IfFileExists{microtype.sty}{\usepackage{microtype}}{}
\ifxetex
  \usepackage[setpagesize=false, 
              unicode=false, 
              xetex]{hyperref}
\else
  \usepackage[unicode=true]{hyperref}
\fi
\hypersetup{breaklinks=true,
            bookmarks=true,
            pdfauthor={},
            pdftitle={},
            colorlinks=true,
            citecolor=blue,
            urlcolor=blue,
            linkcolor=magenta,
            pdfborder={0 0 0}}
\usepackage{amsfonts}
\usepackage{amsmath}
\usepackage{amssymb}
\usepackage{graphicx}

\usepackage{version}%
\usepackage{color}
\usepackage{amscd}     
\usepackage{mathrsfs}  

\newtheorem{theorem}{Theorem}[section]
\theoremstyle{plain}

\newtheorem{corollary}{Corollary}[section]

\newtheorem{definition}{Definition}[section]
\newtheorem{example}{Example}[section]

\newtheorem{proposition}{Proposition}[section]
\newtheorem{remark}{Remark}[section]

\numberwithin{equation}{section}
\excludeversion{comment1}

\newcommand{\st}{\,|\,}

\newcommand{\N}{\mathbb{N}}

\newcommand{\R}{\mathbb{R}}

\newcommand{\cF}{\mathcal{F}}

\newcommand{\cS}{\mathcal{S}}

\newcommand{\Int}{\mathop{\mathrm{int}}}

\newcommand{\be}{\begin{equation}}
\newcommand{\ee}{\end{equation}}

\newcommand{\wt}[1]{{\widetilde{#1}}}

\newcommand{\bP}{{\mathbb{P}}}

\newcommand{\loc}{\mathrm{loc}}
\newcommand{\mydot}{\;\cdot\;}
\DeclareMathOperator{\dom}{dom}
\DeclareMathOperator{\range}{range}
\DeclareMathOperator{\id}{id}
\DeclareMathOperator{\Aff}{Aff}
\DeclareMathOperator{\GL}{GL}
\DeclareMathOperator{\diag}{diag}

\title{Self-referential Functions}
\author{Michael Barnsley \and Markus Hegland \and Peter Massopust}
\date{}

\begin{document}

\begin{abstract}
We introduce the concept of fractels for functions and discuss their analytic and algebraic properties. We also consider the representation of polynomials and analytic functions using fractels, and the consequences of these representations in numerical analysis.
\vskip 6pt
\noindent\textbf{Keywords and Phrases:} Iterated function system (IFS), attractor, fractal function, Read-Bajraktarevi\'c operator, fractel, self-referential function
\vskip 6pt\noindent
\textbf{AMS Subject Classification (2010):} 28A80, 33F05, 41A05, 65D05
\end{abstract}

\maketitle

\section{Introduction}

Contractive operators on function spaces play an important role in the theory of differential and integral equations and are fundamental for the development of iterative solvers. One class of contractive operators is defined on the graphs of functions using a special type of iterated function system (IFS). The fixed point 
of such an IFS is the graph of a fractal function. There is a vast literature on IFSs and fractal functions including the books by the first and third author \cite{B,massopust1,massopust}.

In computer graphics, IFSs are employed in refinement methods which effectively compute points on curves and surfaces~\cite{CavDM91}. They are also used to compute function values of piecewise polynomial functions and wavelets. The graphs of piecewise polynomial functions can be written as the fixed points of (local) IFSs and are thus invariant under the semigroup generated by the maps in the IFS. This invariance induces the well-known self-referentiality of fractal sets. 

In this paper, we introduce and investigate the building blocks of this self-referentiality. These building blocks, called {\em fractels}, play a similar role in an IFS as do finite elements in polynomial approximations of differential equations. It is shown that fractels define IFSs in a natural way and allow the description of vector spaces of functions $[0,1]\to\R$.

The structure of this paper is as follows. In Section 2 we give a brief review of IFSs and their attractors, and define the Read-Bajraktarevi{\'c} operator associated with a particular class of iterated function systems whose attractors are the graphs of functions. Fractels and the concept of self-referentiality are defined in Section 3, where we also collect some algebraic and analytic properties of fractels. The next section focuses on the construction of IFSs from fractels. Affine fractels are then used in Section 5 to describe vector spaces of functions $[0,1]\to\R$ and polynomial bases. In addition, it is shown that they provide efficient algorithms for the evaluation of polynomials overcoming some of the disadvantages of, for instance, Horner's rule.

\section{Brief Review of IFSs and their attractors}

We define an iterated function system (IFS), an attractor of an IFS, and the
basin of an attractor. More information can be found in \cite{BV8}.

\begin{definition}\label{def1}
An \emph{iterated function system (IFS)} is a topological space $X$
together with a finite set of continuous functions $f_{n}:X\rightarrow X$, $n=1,2,\dots,N$.
\end{definition}

We write
\[
\mathcal{F}=\mathcal{F}{(X)}= \{X; f_{1},f_{2},\ldots, f_{N}\}
\]
to denote an IFS. Throughout, $N$ is a finite positive integer and
$X$ is a complete metric space. We use the same symbol, $\mathcal{F}$,
for the IFS and for the set of functions in the IFS.

Let $H=H(X)$ be the collection of nonempty compact subsets of $X$ and
define $\cF:H\rightarrow H$ by
\be\label{F}
\cF(S)=\bigcup\limits_{f\in\mathcal{F}}f(S)
\ee
for all $S\in H$. 

By a slight abuse of notation, we use $\cF$ as the symbol for an IFS and its associated operator. 
We also treat $\cF$ as a map $\cF:2^{X}\rightarrow 2^{X}$, where $2^{X}$ is the collection of all subsets of
$X$. For $S\subset X$, define $\cF^{0}(S)=S$ and let $\cF^{k}(S)$ denote the
$k$-fold composition of $\cF$ applied to $S$, namely, the union of $f_{n_{1}%
}\circ f_{n_{2}}\circ\cdots\circ f_{n_{k}}(S)$ over all finite words
$n_{1}n_{2}\cdots n_{k}$ of length $k$.

Let $d:=d_{X}$ be the metric on $X$, and let $d_{H}%
:=d_{H(X)}$ be the corresponding Hausdorff metric. The Hausdorff
metric on $H(X)$ is defined by%
\[
d_{H}(S,T)=\min\{\delta\geq0 \st S\subset B(T,\delta),\ T\subset B(S,\delta)\}
\]
for all $S,T\in H$, where
\[
B(S,\delta)=\{x\in X \st d_{X}(x,s)\leq\delta,\text{ for some }s\in S\}.
\]
Throughout, the topology on $H$ is the one induced by $d_{H}$. Key facts,
proved in \cite{henrikson}, for example, are that $(H,d_{H})$ is a complete
metric space because $(X,d)$ is complete, and that if $(X,d)$ is
compact then $(H,d_{H})$ is compact.

\begin{definition}
\label{def:attractor} An \emph{attractor} of the IFS $\mathcal{F}$ is a set
$A\in H(X)$ such that

\begin{enumerate}
\item $\cF(A)=A$, and

\item there is an open set $U\subset X$ such that $A\subset U$ and
$\lim_{k\rightarrow\infty}\cF^{k}(S)=A,$ for all $S\in H$ with $S\subset U$,
where the limit is with respect to the Hausdorff metric on $H$.
\end{enumerate}
\end{definition}

The union of all open sets $U$, such that Statement 2 of Definition
\ref{def:attractor} is true, is called the \emph{basin of the attractor }$A$
(with respect to $\mathcal{F}$). If $B$ denotes the basin of $A$, then (it can
be proved that) Statement 2 of Definition \ref{def:attractor} holds with $U$
replaced by $B$. That is, the basin of the attractor $A$ is the largest open
set $U$ such that Statement 2 of Definition \ref{def:attractor} holds.

In much of the discussion in this paper, $\mathcal{F}({X)}$ has a
unique attractor $A,$ and the basin of $A$ is $X$.
\subsection{IFSs and the Read-Bajraktarevi\' c operator}\label{IFSRB}
Suppose we are given a finite family $\{l_n : X\to X \mid n = 1, \ldots, N\}$ of injective contractions with the following two properties: 
\begin{align}
&X = \bigcup_{n=1}^N l_n(X);\label{union}\\
&\Int (l_m(X))\cap \Int(l_n(X)) = \emptyset, \quad\forall\;m, n\in \{1,\ldots, N\}, m\neq n.\label{partition} 
\end{align}
Here, $\Int (S)$ denotes the interior of a set $S$.

Suppose that $(Y,d_Y)$ is a complete metric space with metric $d_Y$. A mapping $g:X\to Y$ is called \emph{bounded} (with respect to the metric $d_Y$) if there exists an $M> 0$ so that for all $x_1, x_2\in X$, $d_Y(g(x_1),g(x_2)) < M$.

Recall that the set $B(X, Y) := \{g : X\to Y \mid \text{$g$ is bounded}\}$ when endowed with the metric 
\[
d(g,h): = \displaystyle{\sup_{x\in X}} \,d_Y(g(x), h(x))
\] 
becomes a complete metric space.

\begin{remark}
Under the usual addition and scalar multiplication of functions, the space $B(X,Y)$ becomes a metric linear space, i.e., a vector space under which the operations of vector addition and scalar multiplication are continuous.
\end{remark}

Let $F_n: X\times Y \to Y$, $n=1, \ldots, N$, be mappings that are uniformly contractive in the second variable, i.e., there exists a $c\in [0,1)$ so that for all $y_1, y_2\in Y$
\be\label{scon}
d_Y (F_n(x, y_1), F_n(x, y_2)) \leq c\, d_Y (y_1, y_2), \quad\forall x\in X,\,\forall n = 1, \ldots, N.
\ee
Define a \emph{Read-Bajractarevi\'c (RB)} $\Phi: B(X,Y)\to B(X,Y)$ by
\be\label{RB}
\Phi g (x) := \sum\limits_{n=1}^N F_n (l_n^{-1} (x), g\circ l_n^{-1} (x))\,\chi_{l_n(X)}(x), 
\ee
where $\chi_M$ denotes the characteristic function of a set $M$. Equivalently, \eqref{RB} can also be written in the form
\be\label{3.3}
(\Phi g \circ l_n) (x) := F_n (x, g(x)),\quad x\in X, \;n = 1, \ldots, N. 
\ee
The operator $\Phi$ is well-defined and since $g$ is bounded and each $v_i$ contractive in the second variable, $\Phi g\in B(X,Y)$.

Moreover, \eqref{scon} implies that $\Phi$ is contractive on $B(X, Y)$:
\begin{align}\label{estim}
d(\Phi g, \Phi h) & = \sup_{x\in X} d_Y (\Phi g (x), \Phi h (x))\nonumber\\
& = \sup_{x\in X} d_Y (F(l_n^{-1} (x), g(l_n^{-1} (x))), F(l_n^{-1} (x), h(l_n^{-1} (x))))\nonumber\\
& \leq c\sup_{x\in X} d_Y (g\circ l_n^{-1} (x), h \circ l_n^{-1} (x)) \leq c\, d_Y(g,h).
\end{align}
We set $F(x,y):= \sum\limits_{n=1}^N F_n (x, y)\,\chi_{X}(x)$ in the above equation. By the Banach Fixed Point Theorem, $\Phi$ has therefore a unique fixed point $f$ in $B(X,Y)$. This unique fixed point is called the \emph{bounded fractal function} (generated by $\Phi$) and it satisfies the \emph{self-referential equation}
\be\label{eq2.8}
f(x) = \sum\limits_{n=1}^N F_n (l_n^{-1} (x), f\circ l_n^{-1} (x))\,\chi_{l_n(X)}(x),
\ee
or, equivalently,
\be
f\circ l_n (x) = F_n (x, f(x)),\quad x\in X, \;n = 1, \ldots, N.
\ee

In case, $Y$ is an \emph{${F}$-space}, i.e., a topological vector space whose topology is induced by a complete translation-invariant metric $d$, and if in addition this metric is also homogeneous, then a special class of mappings $F_n$ may be considered in the definition of $\Phi$. These special mappings are given by
%
\be\label{specialv}
F_n (x,y) := \lambda_n (x) + S_n (x) \,y,\quad n = 1, \ldots, N,
\ee
where $\lambda_n \in B(X,Y)$ and $S_n : X\to \R$ is a function.

As the metric $d_Y$ is homogeneous, the mappings \eqref{specialv} satisfy condition \eqref{scon} provided that the functions $S_i$ are bounded on $X$ with bounds in $[0,1)$. For then
\begin{align*}
d_Y (\lambda_n (x) + S_n (x) \,y_1,&\lambda_n (x) + S_n (x) \,y_2) = d_Y(S_n (x) \,y_1,S_n (x) \,y_2) \\
& = |S_n(x)| d_Y (y_1, y_2) \leq \|S_n\|_{\infty}\, d_Y (y_1, y_2) \leq s\,d_Y (y_1, y_2).
\end{align*}
Here, $\|\cdot\|_{\infty}$ denotes the supremum norm on $\R$ and $s := \max\{\|S_n\|_{\infty}\st$ $n = 1, \ldots, N\}$.

Next, we exhibit the relation between the graph $G(f)$ of the fixed point $f$ of the operator $\Phi$ given by \eqref{RB} and the attractor of an associated contractive IFS. Consider the complete metric space $X\times Y$ and define mappings $w_n:X\times Y\to X\times Y$ by
\be\label{wn}
w_n (x, y) := (l_n (x), F_n (x,y)), \quad n = 1, \ldots, N.
\ee
Assume that the mappings $F_n$ in addition to being uniformly contractive in the second variable are also uniformly Lipschitz continuous in the first variable, i.e., that there exists a constant $L > 0$ so that for all $y\in Y$,
\[
d_Y(F_n(x_1, y),F_n(x_2, y)) \leq L \, d_X (x_1,x_2), \quad\forall x_1, x_2\in X,\quad\forall n = 1, \ldots, N.
\]
Denote by $a:= \max\{a_n\st n = 1, \ldots, N\}$ the largest of the contractivity constants of the $l_n$ and let $\theta := \frac{1-a}{2L}$. Then the mapping $d_\theta : (X\times Y)\times (X\times Y) \to \R$ given by
\[
d_\theta := d_X + \theta\,d_Y
\]
is a metric on $X\times Y$ compatible with the product topology on $X\times Y$.

\begin{theorem}
The family $\cF_\Phi := \{X\times Y; w_1, w_2, \ldots, w_N\}$ is a contractive IFS in the metric $d_\theta$ and the graph $G(f)$ of the fractal function $f$ generated by the RB operator $\Phi$ given by \eqref{RB} is the unique attractor of $\cF_\Phi$. Moreover, 
\be\label{GW}
G(\Phi g) = \cF_\Phi (G(g)),\quad\forall\,g\in B(X,Y),
\ee
where $\cF$ denotes the set-valued operator \eqref{F}.
\end{theorem}
\begin{proof}
For a proof, the reader may consult \cite[Theorem 5.3]{massopust1} applied to the current setting.
\end{proof}

Equation \eqref{GW} can be represented by the following commutative diagram
\be\label{diagram}
\begin{CD}
X\times Y @>\cF_\Phi>> X\times X\\
@AAGA                  @AAGA\\
B(X,Y) @>\Phi>>  B(X,Y)
\end{CD}
\ee

where $G$ is the mapping $B(X,Y)\ni g\longmapsto G(g) = \{(x, g(x))\mid x\in X\}\in X\times Y$.

On the other hand, suppose that $\cF = \{X\times Y; w_1, w_2, \ldots, w_N\}$ is an IFS whose mappings $w_n$ are of the form \eqref{wn} where the functions $l_n$ are contractive injections satisfying \eqref{union} and \eqref{partition}, and the mappings $F_n$ are  uniformly Lipschitz continuous in the first variable and uniformly contractive in the second variable. Then we can associate with the IFS $\cF$ an RB operator $\Phi_\cF$ of the form \eqref{RB}. The attractor $A_\cF$ of $\cF$ is then the graph $G(f)$ of the fixed point $f$ of $\Phi_\cF$. (This was the  approach used by Barnsley \cite{finterp} in the original definition of a fractal interpolation function.) The commutativity of the diagram \eqref{diagram} then holds with $\cF_\Phi$ replaced by $\cF$ and $\Phi$ replaced by $\Phi_\cF$.
\section{Local Iterated Function Systems}\label{sec3}

The concept of \textit{local} iterated function system is a generalization of an iterated function system (IFS). It was first introduced in \cite{barnhurd} and then reconsidered in \cite{BHM}.

\begin{definition}\label{localIFS}
Suppose that $\{X_n \st n = 1, 2, \ldots, N\}$ is a family of nonempty subsets of a Hausdorff space $X$. Further assume that for each $X_n$ there exists a continuous mapping $f_n: X_n\to X$, $n = 1, 2, \ldots, N$. Then the pair $\{X, \cF_{\loc}\}$, where $\cF_{\loc} := \{f_n : X_n\to X \st n = 1, \ldots, N\}$, is called a \emph{local iterated function system (local IFS)}.
\end{definition}

Note that if each $X_n = X$, then we recover Definition \ref{def1} of a standard (global) IFS. The possibility of selecting different domains for the continuous mapping $f_n$ adds additional flexibility. (See, \cite{barnhurd} for applications of this concept.) We note that one may choose the same $X_n$ as the domain for different mappings $f\in \cF_{\loc}$.

One can associate with a local IFS a set-valued operator $\cF_{\loc} : \bP(X) \to \bP(X)$, where $\bP(X)$ denotes the power set of $X$, by setting
\be\label{hutchop}
\cF_{\loc}(S) := \bigcup_{n=1}^N f_n (S\cap X_n).
\ee
By a slight abuse of notation, we use again the same symbol for a local IFS, its collection of functions, and its associated operator.

One can give an alternative definition for \eqref{hutchop}: For given functions $f_n$ that are only defined on
$X_n$, one introduces set functions (also denoted by $f_n$) which are defined on $\bP(X)$ via
\[
f_n (S) := \begin{cases} f_n (S\cap X_n), & S\cap X_n\neq \emptyset;\\ \emptyset, & S\cap X_n = \emptyset,\end{cases}  \qquad n = 1, \ldots, N.
\]
On the left-hand side of the above equation, $f_n (S\cap X_n)$ is the set of values of the original $f_n$ as in the previous definition. This extension of a given function $f_n$ to sets $S$ which include elements which are not in the domain of $f_n$ basically just ignores those elements. In the following we use this definition of the set functions $f_n$. 

\begin{definition}
A subset $A\in \bP(X)$ is called a \emph{local attractor} for the local IFS $\{X, \cF_{\loc}\}$ if
\be\label{attr}
A = \cF_{\loc} (A) = \bigcup_{n = 1}^N f_n (A\cap X_n).
\ee
\end{definition}
In \eqref{attr} it is allowed that $A\cap X_n$ is empty. Hence, every local IFS has at least one local attractor, namely $A = \emptyset$. However, it may also have many distinct ones. In this case, if $A_1$ and $A_2$ are distinct local attractors, then $A_1\cup A_2$ is also a local attractor. Thus, there exists a largest local attractor for $\cF_{\loc}$, namely the union of all distinct local attractors. This largest local attractor is referred to as {\em the} local attractor of a local IFS $\cF_{\loc}$. For more details about local attractors and their relation to the global attractor, we refer the interested reader to \cite{BHM,mass}
\section{Self-referentiality and fractels of functions}
\subsection{Fractels}\label{fractels}
Let $X$ and $Y$ be complete metric spaces, with metrics $d_{X}$ and $d_{Y}$
respectively. Let $l:X\rightarrow X$ be an injection.

\begin{definition}
Let $f: \dom (f) \subseteq X\rightarrow Y$ and let $G(f)$ denote the graph of $f$.  An invertible mapping $w\in C(X\times Y)$ of the form
\begin{equation}
w(x,y)=(l(x),F(x,y)),\label{fractel1}
\end{equation}
for some $F: X\times Y \to Y$,  and which satisfies
\be\label{def:fractel}
w(G(f))\subseteq G(f),
\ee
is called a \emph{fractel} for $f$.
\end{definition}

The name ``fractel'' is a short version of ``fractal element;'' a fractal element
is an element of a set of functions that comprise an IFS of the form
$\mathcal{F}(X\times Y)$. An attractor of a contractive IFS may be a fractal
subset in $\mathbb{R}^{n}$. Contractive fractels for $f:X\rightarrow Y$ may be
used to construct an IFS $\mathcal{F}(X\times Y)$ whose attractor is contained
in, or equal to, $G(f)$. A more general notion of fractel, and applications of
fractels, will be discussed in \cite{fractel}.%

Note that the identity function $\id_{X\times Y}$ on $X\times Y$ is a fractel for any function $f$. We call this fractel the \emph{trivial fractel}.

In the following, we assume that all fractels are nontrivial, unless stated otherwise.

\begin{proposition}\label{subset}
If $w = (l,F)$ is a fractel for $f$, then $l(\dom (f)) \subsetneq \dom (f)$.
\end{proposition}

\begin{proof}
We need to show that $w(x, f(x)) \neq (x, f(x))$, $\forall\,x\in \dom(f)$.  This follows from \eqref{def:fractel} and the nontriviality of $w$.
\end{proof}

\begin{proposition}
\label{Lemma1} The function $w:X\times Y\rightarrow$ $X\times Y$ defined by
(\ref{fractel1}) is a fractel for $f$ if and only if
\begin{equation}
F(x,f(x))=f(l(x)), \label{fractel2}%
\end{equation}
for all $x\in\dom(f) \subseteq X$.
\end{proposition}

\begin{proof}
Suppose (\ref{fractel2}) is true. Then%
\begin{align*}
w(G(f))  &  =\{(x,F(x,f(x)) \mid x\in\dom(f)\subseteq X\}\\
&  =\{(x,f(l(x))) \mid x\in \dom(f)\subseteq X\}\subseteq G(f).
\end{align*}
Conversely, suppose that $w(G(f))\subseteq G(f)$. Then
\[
\{(l(x),F(x,f(x))) \mid x\in\dom(f)\subseteq X\}\subseteq\{(x,f(x)) \mid x\in\dom(f)\subseteq X\}\text{,}%
\]
from which it is immediate that $F(x,f(x))=f(l(x))$, for all $x\in\dom(f)\subseteq X$.
\end{proof}

Notice that different functions may share the same fractel.

\begin{example}\label{ex1}
A nontrivial fractel for $f:[0,\infty)\rightarrow\mathbb{R}$ defined by $f(x):=ax^{p}$,
where $a,p\in\mathbb{R}$, is $w(x,y)=(\frac{x}{2},\frac{y}{2^{p}}).$ Here
$F(x,y)=\frac{y}{2^{p}}$ so that $F(x,a x^{p})=\frac{a x^{p}}{2^{p}}=f(\frac{x}%
{2})$. The fractel $(\frac{x}{2},\frac{y}{2^{p}})$ is independent of $a$ and
linear in $(x,y)$.
\end{example}

\begin{remark}
Suppose $w$ is a fractel for $f:X\to Y$. We consider the following more general setting. Let $D\subset X$ and let $\theta_{D}:2^{X}\rightarrow2^{X}$ be defined by $\theta_{D}(S) := S \cap D$, for all $S\in$ $2^{X}$. We refer then to $w\circ\theta_{D}:2^{X\times Y}\rightarrow2^{X\times Y}$ as a \textbf{local fractel} for $f$. Local fractels for $f:X\rightarrow Y$ may be used to construct a local IFS $\mathcal{F}(X\times Y)$ whose attractor is contained in, or equal to, $G(f)$. For an introduction to local IFSs and their properties, we mention \cite{BHM,mass}.
\end{remark}

Under various conditions, and in various manners, fractels can be composed to
make new fractels, as we illustrate here.

\begin{proposition}
Suppose $w_1(x,y) = (l_1 (x),F_1(x,y))$ and $w_2 = (l_2 (x),F_2(x,y))$ are fractels for $f:X\rightarrow Y$. Then $w_1\circ w_2(x,y) = (l_1\circ l_2 (x),F_1(l_2 (x),F_2(x,y)))$ is also a fractel for $f$.
\end{proposition}

\begin{proof}
We use Proposition \ref{Lemma1}. We have, for all $x\in\dom(f)\subseteq X$,
\[
F_1 (l_2 (x),F_2 (x,f(x))) = F_1 (l_2(x),f(l_2(x)))=f(l_1\circ l_2 (x)).\qedhere
\]
\end{proof}

\begin{corollary}
The collection of all nontrivial fractels of a given function $f$ forms a semi-group under $\circ$. The collection of all fractels of a function $f$ forms a monoid under $\circ$ where the identity element is the trivial fractel.
\end{corollary}

These latter properties give now rise to the following definition.

\begin{definition}
A function $f$ is called \emph{self-referential} if it possesses a nontrivial semigroup of fractels.
\end{definition}

For our later purposes, in particular in the context of computational issues, we now introduce a special class of fractels. 

\begin{definition}
Let $X:= \R^d$ and $Y:= \R^m$. Denote by $\GL(m,\R)$ the general linear group on $\R^m$ and by $\Aff (d, \R) = \R^d \rtimes \GL(d,\R)$ the affine group over $\R^d$. Let $f:\R^d\to \R^m$ be any function. A fractel $w = (l, F)$ of $f$ is called \emph{linear} if
\begin{enumerate}
\item[(i)]	$l\in \Aff(d,\R)$,
\item[(ii)]	$F(x, \mydot)\in \GL(m, \R)$, $x\in \R^d$.
\end{enumerate}
\end{definition}
In other words, a linear fractel $w$ has the explicit form
\be\label{eq.3.4}
w(x,y) = (Ax + b, My),
\ee
where $A$ and $M$ are invertible matrices and $b\in \R^d$. In the case when $d := 1$, we write a linear fractel as
\[
w(x,y) = (\alpha x + \beta, My),
\]
where $\alpha\in\R\setminus\{0\}$ and $\beta\in \R$.

\begin{proposition}\label{prop2.4}
Let $w$ be a linear fractel. Then the matrix $A\in \GL(d,\R)$ has at least one eigenvalue with modulus less than 1. In particular, if $d = 1$, then $|\alpha| < 1$.
\end{proposition}

\begin{proof}
This follows immediately from Proposition \ref{subset} and the nontrivially of $w$. 
\end{proof}

\begin{remark}
Proposition \ref{prop2.4} shows that the collection of all fractels, including the trivial fractel, cannot form a group under function composition as the existence of an inverse element would entail that the matrix $M$ has at least one eigenvalue greater than one in modulus.
\end{remark}

\subsection{Fractels and the RB operator}
We use the notation and terminology from Sections \ref{IFSRB} and \ref{fractels}.
\begin{definition}
Let $w = (l,F)$ be a fractel for the function $f:\dom(f)\subseteq X\to Y$. Then the operator $\Phi_w$ with
\be
\Phi_w f (x) := F(l^{-1}(x), f\circ l^{-1} (x)), \quad x\in \dom(\Phi_w f) = l (\dom (f)),
\ee
is called the \emph{Read-Bajraktarevi\'c (RB) operator of the fractel $w$}.
\end{definition}

\begin{proposition}\label{prop:RB}
If $\Phi_w$ is the RB operator for the fractel $w$, then
\be\label{fractelRB}
w(G(f)) = G(\Phi_w f).
\ee
\end{proposition}

\begin{proof}
We have
\begin{align*}
w (G(f)) &= \left\{(l(x), F(x,f(x)))  \mid  x\in \dom(f)\right\}\\
& = \left\{(x, F(l^{-1}(x), f\circ l^{-1}(x)))  \mid  x\in l(\dom(f)) \right\}\\
& = \left\{(x, \Phi_w f (x))  \mid  x\in \dom(\Phi_w f)\right\} = G(\Phi_w f).\qedhere
\end{align*}
\end{proof}

\begin{example}
Consider the fractels $w(x,y) = (\frac{x}{2}, \frac{y}{2^p} )$ from Example \ref{ex1}. The RB operator $\Phi_w$ of $w$  is given by $(\Phi_w f)(x) = F(2x, f(2x)) = \frac{f(2x)}{2^p}$ for any function $f$.
\end{example}

Next, we state some results that characterize fractels in terms of their RB operator.

\begin{theorem}
Let $f:\dom(f)\subset X\to Y$ be a function. Then, $w = (l,F)$ is a fractel of $f$ iff any of the following statements holds.
\begin{enumerate}
\item[(i)]	$G(\Phi_w f) \subset G(f)$.
\item[(ii)]	$l(\dom(f)) \subset \dom(f)$ and $\Phi_w f (x) = f(x)$, $\forall\,x\in l(\dom(f))$.  
\end{enumerate}
Moreover, $w = (l,F)$ is a linear fractel of $f$ iff $l(\dom(f)) \subset \dom(f)$ and $f(Ax + b) = Mf(x)$, $\forall\,x\in \dom(f)$, where $A$, $M$, and $b$ are the matrices and vector defining $w$ as in \eqref{eq.3.4}.
\end{theorem}

\begin{proof}
Ad (i): If $w$ is a fractel then $w(G(f)) \subset G(f)$. But by Proposition \ref{prop:RB}, $w(G(f)) = G(\Phi_w f)$ and thus $G(\Phi_w f) \subset G(f)$. Conversely, let $\Phi_w f (x) = F(l^{-1}(x), f\circ l^{-1} (x))$, where $x\in \dom(\Phi_w f) = l (\dom (f))$. Set $w := (l,F)$. Then 
\begin{align*} 
w(G(f)) &= \{(l(x), F(x,f(x))  \mid  x \in \dom(f)\}\\
& = \{(x, F(l^{-1}(x), f\circ l^{-1} (x))  \mid  x\in l(\dom(f))\}\\
& =  \left\{(x, \Phi_w f (x))  \mid  x\in \dom(\Phi_w f)\right\} = G(\Phi_w f) \subset G(f).
\end{align*}
Hence, $w$ is a fractel of $f$.

Ad (ii): Suppose $w = (l,F)$ is a fractel for $f$. Then, by Proposition \ref{subset}, $l(\dom(f)) \subset \dom(f)$. Moreover,  
\begin{align*}
\text{$w$ fractel} &\;\Longleftrightarrow\; F(x, f(x)) = f(l(x)), \quad\forall\:x\in \dom(f),\\
& \;\Longleftrightarrow\;  F(l^{-1} (x)), f\circ l^{-1}(x)) = f(x), \quad\forall\;x\in l(\dom(f)),\\
& \;\Longleftrightarrow\;\Phi_w f (x) = f(x), \quad\forall\; x\in l(\dom(f)).
\end{align*}
%

The last statement follows from Propositions \ref{subset} and \ref{Lemma1} with $l(x) = Ax + b$ and $F(x,y) = My$.
\end{proof}
\subsection{Algebra of fractels}
The following discussion introduces ideas related to the existence of
fractels and the algebraic manipulation of fractels. It is relevant to properties
of equivalence classes of functions that share the same fractel.

Propositions \ref{lemma0} and \ref{lemma1b} below provide basic tools for constructing fractels.

\begin{proposition}
\label{lemma0} If $f:X\rightarrow Y$ is bijective, then a fractel for $f$ is $w:X\times Y\mathbb{\rightarrow}X\times Y$ defined by 
\[
w(x,y) := (l(x),f\circ l\circ f^{-1}(y)),\quad\forall\,(x,y)\in X\times Y.
\]
\end{proposition}

\begin{proof}
Write $w(x,y)=(l(x),F(x,y))$ where $F(x,y) := f\circ l\circ f^{-1}(y)$. We have
$F(x,f(x))=f\circ l\circ f^{-1}(f(x))=f(l(x))$, for all $x\in \dom(f)\subseteq X$. The result now follows
from Proposition \ref{Lemma1}.
\end{proof}

We may drop statements such as ``for all $(x,y)\in X\times Y$'' or ``for all $(x,y)\in\dom(f) \times Y$'' when the
context makes the full meaning clear.

\begin{proposition}
\label{lemma1b} Let $T:X\times Y\rightarrow X\times Y$ be of the form
\[
T(x,y)=(T_{1}(x),T_{2}(x,y))
\] 
with inverse $T^{-1}:X\times Y\rightarrow X\times Y$ of the form 
\[
(T_{1}^{-1} (x),T_{2}^{\ast}(x,y)), 
\]
for some surjective function $T_{2}^{\ast}:X\times Y\rightarrow Y$. If $w(x,y)=(l(x),F(x,y))$ is a fractel for $f:X\rightarrow Y$, and $T_{1}\circ l\circ T_{1}^{-1}:X\rightarrow X$ is contractive, then
$T\circ w\circ T^{-1}$ is a fractel for $\widetilde{f}:X\rightarrow Y$ defined
by%
\[
\widetilde{f}(x) := T_{2}(T_{1}^{-1}(x),f(T_{1}^{-1}(x))), \quad\forall\; x\in T_1(\dom(f))\subseteq X.
\]

\end{proposition}

\begin{proof}
The function $T$ maps $G(f)$ into $T(G(f))=\{(T_{1}(x),T_{2}(x,f(x)) \mid x\in\dom(f)\}=\{(x,T_{2}(T_{1}^{-1}(x),f\left(  T_{1}^{-1}(x))\right)  \mid x\in T_1(\dom(f))\} = G(\widetilde{f})$. The function $w$ maps $G(f)$ into $G(f)$. Hence,
$T\circ w\circ T^{-1}$ maps $G(\widetilde{f})$ into $G(\widetilde{f})$. Moreover,
\begin{align*}
T\circ w\circ T^{-1}(x,y)  &  =(T\circ w)(T_{1}^{-1}(x),T_{2}^{\ast}(x,y))\\
&  =T(l(T_{1}^{-1}(x)),F(T_{1}^{-1}(x),T_{2}^{\ast}(x,y)))\\
&  =((T_{1}\circ l\circ T_{1}^{-1})(x),T_{2}(l(T_{1}^{-1}(x)),F(T_{1}%
^{-1}(x),T_{2}^{\ast}(x,y)))).\qedhere
\end{align*}
\end{proof}

Setting in Proposition \ref{lemma1b}, $X := \R^d$, $T_1(x) := A^{-1}(x - b)$, where $A\in \GL(d, \R)$ and $b\in \R^d$, and $T_2 (x,y) := y$, we obtain the following corollary.

\begin{corollary}
If $w = (l(x), F(x,y))$, $l\in \GL(d, \R)$ is a fractel for $f:\R^d\to Y$, then 
\[
(A^{-1}(l (Ax + b)-b), \,F(Ax + b,y))
\]
is a fractel for $f(A \mydot + b)$, where $I$ denotes the unit element of $\GL(d,\R)$.
\end{corollary}

The following Proposition \ref{lemma2} provides a means for constructing fractels for
non-invertible functions $f:X\rightarrow Y$.\ For functions $f,g:X\rightarrow
Y\subset\mathbb{C}^{m}$, $m\in \N$, we use the notation $\left(  f+g\right)  $ to denote
the function defined by $(f+g)(x)=f(x)+g(x)$ where vector addition is implied.
When the inverse of $(f+g)$ is well-defined, we denote this inverse by
$(f+g)^{-1}:Y\rightarrow X$. The technique referred to in Proposition \ref{lemma2}
was used in \cite{continuations} to prove that there exists an analytic IFS
whose attractor is the graph of a given analytic function, where an analytic IFS is
an IFS all of whose functions are analytic.

Throughout, we will state results concerning fractels for functions
$f:X\rightarrow Y\subset$ $\mathbb{C}^{m}$, where $m\in\N$. These results, and their proofs, apply analogously when $\mathbb{C}^{m}$ is replaced by $\mathbb{R}^{m}$.

\begin{proposition}\label{lemma2} 
If $f,g:X\rightarrow Y\subset$ $\mathbb{C}^{m}$ are such that $(f+g)$ is invertible, then a fractel for $f$ is $w:X\times Y\mathbb{\rightarrow}X\times Y$ defined by
\[
w(x,y) := (l(x),(f+g)\circ l\circ(f+g)^{-1}(g(x)+y)-g(l(x)))\text{.}%
\]
\end{proposition}

\begin{proof}
By Proposition \ref{lemma0}, a fractel for $(f+g)$ is defined by 
\[
\widetilde{w}(x,y) = (l(x),(f+g)\circ l\circ(f+g)^{-1}(y)). 
\]
Define $T:X\times\mathbb{R\rightarrow}X\times\mathbb{R}$ by $T(x,y)=(x,g(x)+y)$, so that
$T(G(f))=G(f+g)$. Notice that $T$ is injective, with inverse defined by
$T^{-1}(x,y)=(x,-g(x)+y)$, for all $x$ and $y$. It follows, as in Proposition
\ref{lemma1b}, that a fractel for $f$ is $w=T^{-1}\circ \widetilde{w}\circ T$, which
evaluates to the expression in the statement of the Proposition.
\end{proof}

The set of functions $f:X\rightarrow\mathbb{C}$ comprises a commutative ring with unity $1\neq 0$ over
$\mathbb{C}$, under the usual operations. The following constructions describe
how operations of addition and multiplication on this ring are related to
operations on corresponding fractels (assuming they exist). These
results are useful for constructing fractels. Similar constructions
apply, with care, also to vector-valued functions $f:X\rightarrow\mathbb{C}^{m}$.

\begin{proposition}\label{prop1} Let $f_{1},f_{2}:X\rightarrow\mathbb{C}$. Let $w_{i}:X\times
Y\rightarrow X\times Y$, where $w_{i}(x,y)=(l(x),F_{i}(x,y))$ is a fractel
for $f_{i}$, $i=1,2.$
\begin{enumerate}
\item[(i)] A fractel for $f_{3} := f_{1}+f_{2}$ is $(l(x),F_{3}(x,y))$ where
\[
F_{3}(x,y) = F_{1}(x,y-f_{2}(x))+F_{2}(x,y-f_{1}(x))\text{.}%
\]
\item[(ii)] A fractel for $f_{4} := af_{1}$, where $a\in\mathbb{R}\backslash\{0\}$, is
$(l(x),F_{4}(x,y))$ where%
\[
F_{4}(x,y) = aF_{1}(x,y/a)\text{.}%
\]
\item[(iii)] If $\dom(f_1) \cap \dom(f_2) \neq\emptyset$ and if $f_{1}\cdot f_{2}\neq0$ on $\dom(f_1) \cap \dom(f_2)$, then a fractel for
$f_{5} := f_{1}\cdot f_{2}$ is $(l(x),F_{5}(x,y))$ where%
\[
F_{5}(x,y) = F_{1}\left(x,\frac{y}{f_{2}(x)}\right)\cdot F_{2}\left(x,\frac{y}{f_{1}(x)}\right)\text{.}%
\]
\end{enumerate}
\end{proposition}

\begin{proof}
These results follow at once from Proposition \ref{Lemma1}. \\ \noindent
Ad (i): We have
\begin{align*}
F_{3}(x,f_{3}(x))  &  =F_{1}(x,f_{3}(x)-f_{2}(x))+F_{2}(x,f_{3}(x)-f_{1}(x))\\
&  =F_{1}(x,f_{1}(x))+F_{2}(x,f_{2}(x))\\
&  =f_{1}(l(x))+f_{2}(l(x))=f_{3}(l(x))\text{.}%
\end{align*}

Ad (ii): Similarly%
\[
F_{4}(x,f_{4}(x))=aF_{1}(x,af_{1}(x)/a)=aF_{1}(x,f_{1}(x))=af_{1}%
(l(x))=f_{4}(l(x)).
\]

Ad (iii): Again a direct computation yields
\begin{align*}
F_{5}(x,f_{5}(x))  &  =F_{1}\left(x,\frac{f_{5}(x)}{f_{2}(x)}\right) F_{2}\left(x,\frac
{f_{5}(x)}{f_{1}(x)}\right) = F_{1}(x,f_{1}(x))F_{2}(x,f_{2}(x))\\
&  =f_{1}(l(x))f_{2}(l(x)) = f_{5}(l(x)).\qedhere
\end{align*}
\end{proof}

\begin{example}
If $\mathfrak{R}[x]$ denotes the commutative ring of polynomials $p:[0,1]\rightarrow
\mathbb{R}$, then $\mathfrak{R}[x]$ is generated by $\{1,x\}$. For instance, a fractel for $1$
is $(\frac{x}{2}, 1)$ and a fractel for $x$ is $(\frac{x}2, \frac{y}2)$. By using the
operations in Proposition \ref{prop1}, we can construct a fractel for each
element of $\mathfrak{R}[x]$.
\end{example}

\begin{example}
On the interval $[0,1]$, a fractel for $f_{1}(x)=1$ is $(\frac{x}2,1)$ and a fractel for
$f_{2}(x)=7x$ is $(\frac{x}2, \frac{y}2)$. Here $F_{1}(x,y)=1$ and $F_{2}(x,y)=\frac{y}2$. By
Proposition \ref{prop1}, a fractel on $[0,1]$ for $f(x)=10+x=10\cdot f_{1}(x)+\frac{1}%
{7}\cdot f_{2}(x)$, so that $a_{1}=10$ and $a_{2}=\frac{1}{7}$, is
\[
w(x,y)=\left(\tfrac{x}2, 5+\tfrac{y}2\right)\text{.}%
\]
\end{example}
\subsection{Fractels, cartesian products, and function compositions}
In this short section, we consider fractels for cartesian products and compositions of functions. First, we recall that the cartesian product of two functions $f_1: \dom(f_1) \subseteq X_1 \to Y_1$ and $f_2: \dom(f_2) \subseteq X_2 \to Y_2$ is defined to be the function 
\begin{gather*}
f_1\times f_2: \dom(f_1)\times\dom(f_2) \subseteq X_1\times X_2 \to Y_1\times Y_2,\\ 
f_1\times f_2 (x_1, x_2) := (f_1(x_1), f_2(x_2)).
\end{gather*}

Suppose that for $i = 1,2$, $X_i$ and $Y_i$ are complete metric spaces and $f_i: \dom(f_i)\subseteq X_i\to Y_i$ are functions with fractels $(l_i, F_i)$, where $l_i : X_i\to X_i$ and $F_i : X_i\times Y_i\to Y_i$. Then a fractel for $f_1\times f_2: \dom(f_1)\times\dom(f_2) \subseteq X_1\times X_2\to Y_1\times Y_2$ is given by $(l_1\times l_2, F_1\times F_2)$:
\begin{align*}
f_1\times f_2 (l_1\times l_2 (x_1,x_2)) &= (f_1(l_1(x_1)),f_2(l_2(x_2))) = (F_1(x_1,f_1(x_1)),F_2(x_2,f_2(x_2)))\\
&= F_1\times F_2 ((x_1,f_1(x_1)),(x_2,f_2(x_2))),
\end{align*}
for all $(x_1,x_2)\in \dom(f_1)\times\dom(f_2)$.

In the case that $X_1 = X_2 =: X$, we identify $\diag (X\times X)$ with $X$ and consider $f_{1}: \dom(f_1)\subseteq X\rightarrow Y_1$ and $f_{2}: \dom(f_2)\subseteq X\rightarrow Y_2$. Assume that $(l(x),F_{1}(x,y_1))$ and $(l(x),F_{2}(x,y_2))$ are fractels for $f_{1}$ and $f_{2}$, respectively. Then, a fractel for $f_1\times f_2:X\to Y_1\times Y_2$, $x\mapsto (f_1(x),f_2(x))$, is given by $(l(x),F_{1}(x,y_1)\times F_{2}(x,y_2))$, where
\[
F_{1}(x,y_1)\times F_{2}(x,y_2) := (F_1(x,y_1),F_2(x,y_2)), \quad\forall\,x\in X, \,y_1, y_2\in Y.
\]
Indeed, 
\begin{align*}
f_1\times f_2 (l(x)) &= (f_1(l(x)),f_2(l(x))) = (F_1(x,f_1(x))), F_2(x,f_2(x)))\\
& = F_{1}(x,f_1(x))\times F_{2}(x,f_2(x)),\qquad \forall\,x\in \dom(f_1) \cap \dom(f_2).
\end{align*}

As far as function compositions are concerned, we have the following result. Given that $f_1:X\to Y$ and $f_2: Y \to Z$ are functions such that $\range (f) = \dom (g)$ and which have fractels $(l_1,F_1)$ and $(l_2,F_2)$, then a fractel for $f_2\circ f_1:X\to Z$ is given by $(l_1, F_2)$ provided that $f_1(l_1(x)) = F_1(x, f_1(x)) = l_2(x)$. This follows directly from the application of Proposition \ref{Lemma1}.

\subsection{Calculus of fractels}
In some cases it is possible to derive fractels for integrals and derivatives
of functions from fractels for functions. In these cases, we solve such
problems as: ``Given $w(x,y)$ is a fractel for $f:[0,1]\rightarrow\mathbb{R}$,
find fractels for (i) $g(x)=\int\limits_{0}^{x}f(t)dt$ and (ii) $h(x)=\frac
{df}{dx}(x)=f^{\prime}(x)$.'' This approach generalizes a basic idea, exploited in
\cite{finterp2}, concerning the calculus of fractal interpolation functions.
Propositions \ref{lemma3} and \ref{lemma4} are illustrative; further results,
in the context of higher dimensions, with more elaborate dependencies on $x$
and $y$, may be obtained.

\begin{proposition}
\label{lemma3} Let both $f:[0,1]\rightarrow\mathbb{R}$ and $g:[0,1]\rightarrow
\mathbb{R}$ be differentiable. Let $l(x)=sx$, for some $s\in (0,1)$. If
$w(x,y)=(sx,g(x)+cy)$ is a fractel for $f:[0,1]\rightarrow\mathbb{R}$, for
some $c\in\mathbb{R}$, then a fractel for $f^{\prime}:[0,1]\rightarrow
\mathbb{R}$ is $w=(sx,F(x,y))$ where%
\[
F(x,y)=s^{-1}g^{\prime}(x)+s^{-1}cy\text{.}%
\]

\end{proposition}

\begin{proof}
Eqn. \eqref{fractel2} implies $f(sx)=g(x)+cf(x)$, from which it follows that
\begin{align*}
f^{\prime}(sx)  &  =s^{-1}g^{\prime}(x)+s^{-1}cf^{\prime}(x) =F(x,f^{\prime}(x)),
\end{align*}
for all $x\in(0,1)$. The result now follows from Proposition \ref{Lemma1}.
\end{proof}

\begin{example}
A fractel for $f(x)=x^{3}$ is $(\frac{x}{2},\frac{y}{8})$. It follows from
Proposition \ref{lemma3} that a fractel for $f^{\prime}(x)=3x^{2}$ is $(\frac{x}%
{2},\frac{y}{4})$.
\end{example}

The generalization of Proposition \ref{lemma3} to the case where $f,g\in C^n[0,1]$ and $L := {\sum\limits_{i=1}^n}\, a_i(x) D^i$ is a linear differential operator with $a_i \in C[0,1]$ is at hand.

We can similarly obtain fractels for $\int\limits_{0}^{x}f(t)dt$ from fractels
for $f(x)$, in some cases. See also \cite{finterp2}.

\begin{proposition}
\label{lemma4} Let both $f:[0,1]\rightarrow\mathbb{R}$ and $g:[0,1]\rightarrow
\mathbb{R}$ be Riemann integrable. Let $l(x)=sx$, for some $s\in (0,1)$. If
$w(x,y)=(sx,g(x)+cy)$ is a fractel for $f:[0,1]\rightarrow\mathbb{R}$, for
some $c\in\mathbb{R}$, then a fractel for $\int\limits_{0}^{x}%
f(t)dt:[0,1]\rightarrow\mathbb{R}$ is $w=(sx,F(x,y))$ where%
\[
F(x,y):=s\int\limits_{0}^{x}g(t)dt+scy\text{.}%
\]
\end{proposition}

\begin{proof}
From $f(s x)=g(x)+cf(x)$, it follows that
\begin{align*}
\int\limits_{0}^{sx}f(t)dt = s \int\limits_{0}^{x}f(st)dt = s\int\limits_{0}^{x}g(t)dt+s c \int
\limits_{0}^{x}f(t)dt,
%
%
\end{align*}
for all $x\in(0,1)$. The result now follows from Proposition \ref{Lemma1}.
\end{proof}

\begin{example}
A fractel for $f(x)=x^{3}$ is $(\frac{x}{2},\frac{y}{8})$. It follows from
Proposition \ref{lemma4} that a fractel for $\int\limits_{0}^{x}f(t)dt=\frac{x^{4}%
}{4}$ is $(\frac{x}{2},\frac{y}{16})$.
\end{example}

Using Cauchy's formula for repeated integration, one can extend Proposition \ref{lemma4} to fractional integrals. Suppose $f, g\in L^1 (0,b)$, for some $0<b\leq\infty$. Then, for $\alpha > 0$, the $\alpha$-fractional integral of $f$, $J^\alpha f$, is defined by
\be\label{J}
J^\alpha f (x) := \frac{1}{\Gamma (\alpha)}\,\int_0^x (x - t)^{\alpha -1} f(t) dt, \quad x\leq b.
\ee
Under the above hypotheses on $f$ and $g$, if $w(x,y)=(sx,g(x)+cy)$ is a fractel for $f:[0,1]\rightarrow\mathbb{R}$, then an argument similar to the one employed in the proof of Proposition \ref{lemma4}, yields 
\[
w(x,y) = (s x, s^\alpha\,J^\alpha g (x) + s^\alpha\,c y)
\]
as a fractel for $J^\alpha f$.

Note that for $\alpha := \frac12$, \eqref{J} is related to the Abel integral equation.

\section{\label{sec:1D} Fractels and IFSs for real functions}

In this section, we develop fractels and from them local IFSs for functions $f: \Omega \rightarrow \R$
for some closed interval $\Omega \subset \R$. The fractels are found using Proposition \ref{Lemma1} which states that 
$w(x,y)=(l(x),F(x,y))$ is a fractel of a function $f$ if and only if $$f(l(x)) = F(x,f(x)).$$
These fractels will then be used to establish fractal approximations of $f$. 

We consider the case of fractels which have a domain $\dom w = \dom l \times \R$ where the domain of $l$ is an interval 
$[a,b]$. The function $l$ is chosen as 
$$l(x) = \frac{x + \tau}{2}$$
for some $\tau\in[a,b]$ with typical values $\tau=a$, $\tau=b$ or $\tau=(a+b)/2$. The function $F(x,y)$ is chosen as
$$F(x,y) = \sigma y + (1-\sigma)G(x)$$
for some $0\leq \sigma < 1$ and some continuous function $G$.
Approximations are obtained by approximating $G(x)$; here we will consider constant approximations $G(x)\approx \gamma$. 
The functions 
$$w(x,y) = ((x+\tau)/2, \sigma y + (1-\sigma)\gamma)$$ 
are fractels for functions $h$ of the form
$$h(x) = \alpha (x-\tau)^\theta + \gamma$$
where $2^{-\theta} = \sigma$. This can be shown by applying Proposition \ref{Lemma1}.

We can now generate an approximation of a function $f(x)$ by first constructing its fractels (and local IFS) of the form
$w(x,y) = ((x+\tau)/2, \sigma y + (1-\sigma) G(x))$ and then obtaining an approximation by approximating $G(x)$ by a constant
$\gamma$ which we suggest can be done in three ways (among many others of course):
\begin{itemize}
  \item by the mean $$\gamma = \frac{1}{b-a}\int_a^b G(x)\, dx;$$
  \item by the midpoint $$\gamma = G((a+b)/2);$$
  \item or by the trapezoidal rule for the mean $$\gamma = (G(a)+G(b))/2.$$
\end{itemize}
Using a bound on the errors $G(x)-\gamma$ for all the fractels of an IFS one can obtain a bound for the approximation error. The argument is based on the following observation. Suppose $k$ and $\wt{k}$ satisfy self-referential equations of the form \eqref{eq2.8} where $F$ is given by \eqref{specialv}:
\begin{gather*}
k(x) = \sum_{n=1}^N \lambda_n(l_n^{-1}(x)) \chi_{l_n(\Omega)} + \sum_{n=1}^N s_n k (l_n^{-1}(x)) \chi_{l_n(\Omega)}, \\
\wt{k}(x) = \sum_{n=1}^N \wt{\lambda}_n(l_n^{-1}(x)) \chi_{l_n(\Omega)} + \sum_{n=1}^N s_n \wt{k} (l_n^{-1}(x)) \chi_{l_n(\Omega)},
\end{gather*}
with $s_n\in (-1,1)$. Using arguments similar to those given in \cite[5.11]{massopust}, one shows that
\be\label{eq5.1}
\|k - \wt{k}\|_\infty \leq \frac{\max\limits_{n = 1, \ldots, N} \|\lambda_n - \wt{\lambda}_n\|_\infty}{1-s},
\ee
where $s:= \max\limits_{n = 1, \ldots, N} |s_n|$.

Now suppose that $f:\Omega\to \R$ has fractel $w(x,y) = ((x+\tau)/2, \sigma y + \lambda(x))$, where $\lambda$ is given by $(1-\sigma)G$. Further suppose that $\wt{\lambda} = (1-\sigma)\gamma$ is an approximation for $\lambda$. Denote by $\wt{f}$ the function with fractel $\wt{w}(x,y) = ((x+\tau)/2, \sigma y + \wt{\lambda}(x)) = ((x+\tau)/2, \sigma y + (1 - \sigma) \gamma))$. Then, by the above observation, we have
\[
\|f - \wt{f}\|_\infty \leq \|G - \gamma\|_\infty.
\]

We have the following proposition.
\begin{proposition}
   \label{prop1D}
   A function $f$ defined by
    $$f(x) = \alpha (x-\tau)^\theta + g(x),$$
   for $x\in[a,b]$ with $\theta >0$ and $\tau\in[a,b]$, admits a fractel $w$ with
    $$w(x,y) = ((x+\tau)/2, \sigma y + (1-\sigma) G(x)),$$
   where $\sigma = 2^{-\theta}$ and 
    $$G(x) = \frac{g((x+\tau)/2)-\sigma g(x)}{1-\sigma}.$$
\end{proposition}
\begin{proof}
  This is a direct application of Proposition \ref{Lemma1}.
\end{proof}

There are now some special cases:
\begin{itemize}
  \item if $g(x)=\gamma$ then $G(x) = \gamma$, independent of $\theta$,
  \item if $\theta=1$ and $g(x)$ is a polynomial of degree one, then $G(x)=g(\tau)$,
  \item if $\theta=2$ and $g(x)$ is a polynomial of degree one, then $G(x) = g((x+2\tau)/3)$.
  \item if $\alpha=0$ then $f(x)=g(x)$ and if we now choose $\sigma=0$ one gets
    $$G(x) = f((x+\tau)/2).$$
    This is the case which recovers standard approximations and the IFS only is used to control the evaluation.
\end{itemize}

We now illustrate the above discussion for a simple example by developing a full local IFS and a corresponding
approximation.

\begin{example}
    We consider the function
    $$f(x) = \sqrt{x}, \quad x\in [0,1].$$
    By Proposition \ref{prop1D} this function admits a fractel 
    $$w_1(x,y) = (x/2,\, y/\sqrt{2})$$
    with domain $[0,1]\times\R$. This fractel recovers the function $f(x)$ for $x\in[0,\frac12]$ from the full function.

    We now need to recover the function for $x\in(\frac12,1]$. For this we introduce two fractels defined over the
    domain $[\frac12,1]\times \R$. First we choose the components $l_i$ as
    $$l_i(x) = (x + \tau_i)/2$$ 
    where $\tau_i = (i-1)/2$. Then we rewrite our fractel as
    $$f(x) = \alpha_i(x-\tau_i)^{\theta_i} + g_i(x)$$
    thus
    $$g_i(x) = \sqrt{x} - \alpha_i(x-\tau_i)^{\theta_i},$$
    we are free to choose the $\theta_i > 0$ and we set as above $\sigma_i = 2^{-\theta_i}$.
  
    We now use Proposition \ref{prop1D} to get 
    \begin{align*}
      G_i(x) & = \frac{g_i((x+\tau_i)/2)-\sigma_i g_i(x)}{1-\sigma_i}\\
             & = \frac{\sqrt{(x+\tau_i)/2} - \sigma_i \sqrt{x} -\alpha_i\left(
                         \sigma_i (x-\tau_i)^\theta_i - \sigma_i (x-\tau_i)^{\theta_i}\right)}{1-\sigma_i}
    \end{align*}
    and thus
    $$G_i(x) = \frac{\sqrt{(x+\tau_i)/2} - \sigma_i \sqrt{x}}{1-\sigma_i}.$$

    In summary, this leads to a local IFS defined by
    \begin{align}
       w_1(x,y) & = (x/2, y/\sqrt{2}), \quad x \in [0,1], \\
       w_2(x,y) & = ((2x+1)/4, \sigma_2 y + \sqrt{2x+1}/2 - \sigma_2 \sqrt{x}), \quad x\in[\tfrac12,1], \\
       w_3(x,y) & = ((x+1)/2, \sigma_3 y  + \sqrt{x+1}/\sqrt{2} - \sigma_3 \sqrt{x}), \quad x\in[\tfrac12,1].
    \end{align}

    In this case the function $f(x)$ is smooth for $x\in[\frac12,1]$ and one typically would choose $\theta_i \geq 1$ and
    thus $\sigma_i \in [0,\frac12]$. 

    We then get an approximation by the midpoint rule to be for $\sigma_i=\frac12$:
    \begin{align}
       w_1(x,y) & = (x/2, y/\sqrt{2}), \quad x \in [0,1], \\
       w_2(x,y) & = ((2x+1)/4, (y + \sqrt{5/2} -  \sqrt{3/2})/2, \quad x\in[\tfrac12,1], \\
       w_3(x,y) & = ((x+1)/2,  (y + \sqrt{5} - \sqrt{3/2})/2, \quad x\in[\tfrac12,1].
    \end{align}

    \begin{figure}
        \centerline{\includegraphics[width = 0.9\textwidth]{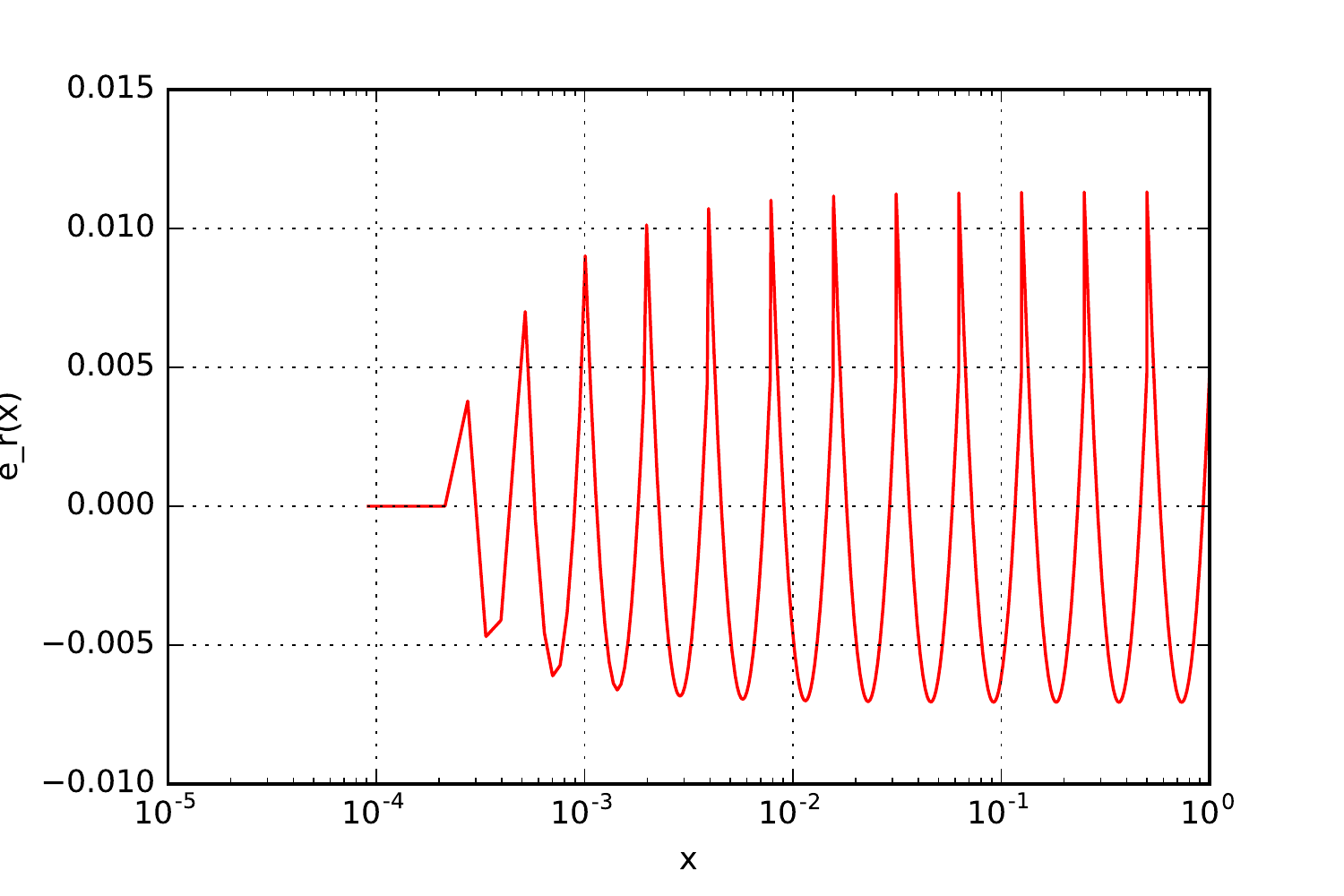}}
        \caption{\label{fig:relerr} 
            The relative error of the fractal approximation of $f(x) = \sqrt{x}$ based on the midpoint formula.}
    \end{figure}
    In Figure \ref{fig:relerr} we have plotted the relative error $e(x) = f_F(x)/f(x) - 1.0$ where $f(x)=\sqrt{x}$
    and $f_F(x)$ is the fractal approximation obtained from the midpoint formula above. One can see that the relative
    error is stable even for $x$ close to zero. One can show that the approximation is actually a piecewise linear
    approximation for a grid which is refined at $x=0$.

\end{example}

In this section we have considered fractels and corresponding local IFSs based on a particular choice of the functions
$F(x,y) = \sigma y + G(x)$. We have seen that for any $\sigma\in[0,1)$ we can generate a fractel. Of course it makes
sense to choose the $\sigma$ such that it reflects the main (local) behaviour of the function. For a local IFS it is
important to choose the domain carefully. This needs to be further discussed. Here we have provided an example. Future
work would include the discussion of fractels of the form $\Phi(y) + G(x)$ and how to choose appropriate $\Phi(y)$. 
One can see again that $\Phi(y)$ may in principle be chosen arbitrarily, but like in this section, we suggest to take
into account approximation techniques when choosing $\Phi(y)$. For example, one might choose $\sigma$ and $\Phi$ such
that a constant approximation of the corresponding $G(x)$ gives a good approximation. 

\section{Affine Fractels For Vector Spaces of Functions
$u:[0,1] \rightarrow \R$}\label{affine-fractels-for-vector-spaces-of-functions-u01-rightarrow-r}
In the following, let $V_n$ denote an $m$-dimensional linear space of functions $u: [0,1]\rightarrow \R$. Furthermore, let 
\[
\cS=\{l  \mid  l(x) = \sigma x+\tau,\; l([0,1]) \subset [0,1]\}
\] 
be the set of affine maps which leave the interval $[0,1]$ invariant. One can show that $\cS$ forms a semigroup with respect to composition. As the affine maps are fractels for the interval $[0,1]$, we call $\cS$ a
fractel semigroup.

\begin{definition}
We refer to
\[
\cS_{V_n} := \{l\in \cS  \mid  u\circ l \in V_n,\; \text{for all $u\in V_n$}\}.
\]
as the \emph{fractel semigroup} of $V_n$.
\end{definition}

We thus have for every $l\in \cS_{V_n}$ 
\[
l^*(V_n) \subset V_n,
\] 
where $l^*$ is the pullback of $l$, i.e., $l^*(u) = u\circ l$. As the fractel semigroup always contains the identity mapping $l(x) = x$, one has that
\[
V_n = \bigcup_{l\in \cS_{V_n}} l^*(V_n).
\]

If $V_n$ is the set of polynomials of degree at most $n-1$ then $\cS_{V_n}=\cS$. If $V_n$ is the space of continuous piecewise linear functions which are linear on $[0,\frac12]$ and $[\frac12,1]$ then $\cS_{V_n}$ contains all $l\in\cS$ which satisfy one of the three following conditions: 
\begin{enumerate}
\item $l([0,1]) \subset [0,\frac12]$,
\item $l([0,1]) \subset [\frac12,1]$, or 
\item $l(\frac12)=\frac12$,
\end{enumerate}
as in the first two cases $l^*(u)$ are first degree polynomials and in the third case the corner is invariant at $x=\frac12$. In this case one has $l(x) = \sigma x + \frac12\,(1-\sigma)$ and only in this case one has $l^*(V_n)=V_n$ as in the other two cases $l^*(V_n)$ is the set of affine functions which is a subset of $V_n$.

The following proposition is a direct consequence of the definitions.

\begin{proposition}
Let $f_1,\ldots,f_n$ be basis functions of $V_n$ and $f=(f_1,\ldots,f_n)^T$. Let $l\in \cS_{V_n}$ be an element of the affine fractel semigroup of $V_n$. Then
  \begin{itemize}
    \item there exists a matrix $M_l \in \R^{n\times n}$ such that
       \[
       f\circ l = M_l f.
       \]
    \item If $l^*(V_n) = V_n$ then $M_l$ is invertible.
  \end{itemize} 
\end{proposition}

\begin{proof}
As $f_i\circ l \in V_n$ it follows that there exist $m_{i,1},\ldots,m_{i,n}$ such that
\[
f\circ l = \sum\limits_{i=1}^n m_{i,k} f_k.
\]
If $l^*(V_n) = V_n$ then the $l^*(f_i)$ is also a basis of $V_n$ and thus $M$ is invertible.
\end{proof}

A consequence of this proposition is that one can derive the fractels of a vector of basis functions of $V_n$ from the affine fractel semigroup.

\subsection{Fractels of generators of function spaces}\label{fractels-of-generators-of-function-spaces}

In this section, we will consider linear fractels for vectors $f$ of real functions defined on $[0,1]$. The components $f_j$ of this vector span a linear space $V_n$ of real functions and, conversely, for every
linear space $V_n$ of functions one can find a vector of generating functions. 

The linear fractels $w:[0,1]\times \R^m \rightarrow [0,1]\times\R^m$ considered here are of the form 
\[
w(x,y) = (l(x), My), 
\] 
where $M\in \R^{m\times m}$ and $l(x) = \sigma x + \tau$, and either $\sigma\in(0,1]$ and $\tau\in[0,1-\sigma]$ or $\sigma\in[-1,0)$ and $\tau\in[|\sigma|,1]$, respectively. Note that in both cases $\tau\in [0,1]$ and $\sigma\neq 0$.

Recall that a function $f:[0,1] \rightarrow \R^m$ admits a fractel defined by $w(x,y)$ of this form if $w$ maps the graph $G(f) = \{(x,f(x)  \mid  x \in [0,1]\}$ onto itself, i.e.
\[ 
\{(l(x), Mf(x)) \st x \in [0,1]\} \subset G(f),
\] 
or, equivalently,
\[ 
f\circ l(x) = M f(x), \quad x \in [0,1]. 
\]

We will now collect some properties of linear fractels of vectors of
functions.

\begin{proposition}
Let $w(x,y)=(\sigma x+\tau, My)$ define a fractel of a function $f:[0,1]\rightarrow \R^m$. Then
\begin{itemize}
  \item There exists an $x^*\in [0,1]$ such that $l(x^*) = x^*$.
  \item Furthermore, $Mf(x^*) = f(x^*)$.
  \item Either $f(x^*)=0$ or $f(x^*)$ is an eigenvector of $M$ with eigenvalue 1.
  \item If $f(x^*)\neq 0$ then there exists a $c\in \R^m$ such that $u_0(x):= c^T f(x)$ satisfies
  \[
  u_0(\sigma x + \tau) = u_0(x), \quad x \in [0,1].
  \]
  \item If the function $f$ is continuous and $f(x^*)\neq 0$ then $u_0(x)$ is constant with value $u_0(x)=u_0(x^*)$.
\end{itemize}
\end{proposition}

\begin{proof}
As $w$ defines a fractel one has $l([0,1]) \subset [0,1]$ and by the Brouwer fixpoint theorem there exists a fixpoint $x^*$. 

As $f$ admits the fractel defined by $w$ one has $ Mf(x^*) = f(l(x^*))$ from which the second claim follows. 

For the fourth claim choose $c^T$ to be a real left eigenvector of $M$ for the eigenvalue 1 and apply the fractel to $f$ to get $c^T f(x) = c^T M f(x) = c^T f(l(x))$.

Finally, consider for any $x$ the sequence given by $x_0 = x$ and $x_{k+1} = \sigma x_k + \tau$. One has $u_0(x_k) = u_0(x)$ and, as the sequence converges to the fixpoint $x^*$ one thus gets by continuity of $f$
that $u_0(x) = u_0(x^*)$.
\end{proof}

The components $f_j$ of the vector $f$ define a vector space $V$ of functions $u:[0,1]\rightarrow \R$ by
\[ 
V = \{ u(x) = c^T f(x)  \mid c \in \R^m \}. 
\] 
A common assumption made for function spaces is that they contain constant functions. This property of $V$ follows from continuity of the elements of $V$ and the fact that it admits a fractel.

\begin{corollary}
Let $V$ be an $m$-dimensional space of continuous functions which are such that for each $x\in[0,1]$ there exists a $u\in V$ such that $u(x)\neq 0$. If $V$ admits a fractel $w(x,y)=(\sigma x + \tau, My)$ then $V$ contains the constant functions.
\end{corollary}

The condition that not all $u$ are zero at the same point $x_0$ is essential for the solution of interpolation problems: A real function $u(x)$ with $u(x_0)\neq 0$ cannot be interpolated in $V_n$ if $x_0$ is an interpolation point.

One then gets directly for the case $m=1$:

\begin{corollary}
The only functions $f:[0,1]\rightarrow \R$ with $f(x) \neq 0$ for all $x$ which admit a fractel defined by $w(x,y)=(\sigma x+\tau, \mu y)$ are functions which admit the symmetry
\[
f(\sigma x + \tau) = f(x)
\]
and $\mu=1$. The only continuous real functions which admit a fractel are the constants.
\end{corollary}

\begin{proposition}
If $e^T f(x) = 1$ for $e=(1,1,\ldots,1)^T$ then the matrix $M$ of the fractel of $f(x)$ is stochastic if its elements are $m_{i,j} \geq 0$.
\end{proposition}

\subsection{Fractels for a polynomial basis}\label{fractels-for-a-polynomial-basis}

Here we consider the space $P_k$ of polynomials of degree at most $k$. In this case the fractel semigroup is the full fractel semigroup of $[0,1]$. Our main result is a characterization of all fractels of all vectors of basis functions of $P_k$.

\begin{proposition}
The fractel semigroup of the space of polynomials $P_k$ of degree at most $k$ is
\[
\cS_{P_k} = \cS.
\]
\end{proposition}

\begin{proof}
As $l(x)=\sigma x + \tau$ it follows that for any polynomial $p$ the composition $p\circ l$ is a polynomial of the same degree if $\sigma\neq 0$. The claim follows then directly from the definition of the fractel semigroup.
\end{proof}

\begin{proposition}
Any fractel of a vector $f$ of basis polynomials of $P_k$ is of the form
\[
w(x,y) = (l(x), T M_l T^{-1} y)
\]
where
  \begin{itemize}
    \item $l\in \cS$,
    \item the matrix $T\in GL(k+1,\R)$ satisfies $f(x) = T (1,x,\ldots,x^k)^T$,
    \item the matrix $M_l$ has the components
     \[
     m_{s,t} = \begin{cases} \binom{s}{t} \tau^{s-t}\sigma^t,\quad & 0 \leq t \leq s \leq k; \\ 0, & \text{else},	\end{cases}
     \]
      where $l(x) = \sigma x + \tau$.
  \end{itemize}
\end{proposition}

\begin{proof}
Let $f$ be a vector of basis functions of $P_k$. Then there exists a matrix $T\in GL(k+1,\R)$ such that $f(x)=Tg(x)$ where $g(x)$ is a vector with components $g_s(x) = x^s$, for $s=0,\ldots,k$. As
\[
g_s(\sigma x + \tau) = \sum\limits_{t=0}^s \binom{s}{t} \tau^{s-t} \sigma^t x^t
\]
one obtainss $g(\sigma x + \tau) = M_l g(x)$ and thus
\begin{align*}
    f(\sigma x + \tau) &= T g(\sigma x + \tau) \\
                       &= T M_l g(x) \\
                       &= T M_l T^{-1} f(x).
\end{align*}
It follows now that $w(x,y)=(l(x), TM_l T^{-1}y)$ is a fractel of $f(x)$.

Conversely, let $w(x,y) = (l(x), My)$ be any fractel of $f(x)$. Then $ f(\sigma x + \tau) = M f(x)$ and thus $M f(x) = T M_l T^{-1} f(x)$, for all $x\in[0,1]$. As $f$ is a vector of basis functions it follows that $M= TM_l T^{-1}$.
\end{proof}

We will now illustrate the previous results by some examples.

\begin{example}[Monomial Basis: $f(x) = (1,x,x^2,x^3)^T$]
Here one has $T=I$ and $w(x,y) = (\sigma x + \tau, M_l y)$ with
\[
M_l = \begin{bmatrix} 1 & & & \\ \tau & \sigma & & \\ \tau^2 & 2\sigma \tau & \sigma^2 & \\
                                     \tau^3 & 3 \sigma \tau^2 & 3 \sigma^2 \tau & \sigma^3 \end{bmatrix}.
\]
For practical computations are important $l_1(x)= \frac{x}{2}$ and $l_2(x) = \frac{x+1}{2}$. where
\[ 
M_{l_1} = \begin{bmatrix}1 & & & \\[4pt] & \frac{1}{2} & & \\[4pt] & & \frac{1}{4} & \\[4pt] & & & \frac{1}{8} \end{bmatrix}, 
\quad 
M_{l_2} = \begin{bmatrix}1 & & & \\[4pt] \frac{1}{2} & \frac{1}{2} & & \\[4pt] 
  \frac{1}{4} & \frac{1}{2} & \frac{1}{4} & \\[4pt] \frac{1}{8} & \frac{3}{8} & \frac{3}{8} & \frac{1}{8} \end{bmatrix}. \]
\end{example}

\begin{example}[Degree One Polynomials and $f(x) = (1-x, x)^T$]
In this case one has for the transformation matrix 
\[
T = \begin{bmatrix} 1 & -1 \\ 0 & 1 \end{bmatrix}
\]
and for the matrix $M=TM_l T^{-1}$:
\[
M = \begin{bmatrix} 1 - \tau & 1 - \tau -\sigma \\ \tau & \tau + \sigma \end{bmatrix}.
\]
Note that as $e^T f(x) = 1$ the matrices $M$ are stochastic. 

Using the same affine functions $l_1$ and $l_2$ as in the previous example, one obtains for $M_i = T M_{l_i} T^{-1}$:
\[
M_1 = \begin{bmatrix} 1 & \frac{1}{2} \\[4pt] 0 & \frac{1}{2} \end{bmatrix}, \quad
M_2 = \begin{bmatrix}  \frac{1}{2}  & 0 \\[4pt] \frac{1}{2} & 1 \end{bmatrix}.
\]
\end{example}

\begin{example}[3rd Degree Chebyshev Polynomials: $f(x)=(1,x,2x^2-1,4x^3-3x)^T$]
Here one can compute the transformation matrix $T$ to be
\[
T = \begin{bmatrix}1&0 & 0& 0 \\0 & 1 &0 &0 \\ -1 & 0 & 2 & 0 \\ 0 & -3 & 0 & 4 \end{bmatrix}.
\]
Here we only consider the two special cases for the $l_i$ used previously and $M_i= T M_{l_i} T^{-1}$:
\[
M_1 = \begin{bmatrix} 1 & 0 & 0 & 0 \\[4pt] 0 & \frac{1}{2} & 0 & 0 \\[4pt] -\frac{3}{4} & 0 & \frac{1}{4} & 0 \\[4pt] 0 & -\frac{9}{8} & 0 & \frac{1}{8} \end{bmatrix}, 
\quad 
M_2 = \begin{bmatrix} 1 & 0 & 0 & 0 \\[4pt] \frac{1}{2} & \frac{1}{2} & 0 & 0 \\[4pt] -\frac{1}{4} & 1 & \frac{1}{4} & 0 \\[4pt] -\frac{1}{4} & \frac{3}{8} & \frac{3}{4} & \frac{1}{8} \end{bmatrix}.
\]
\end{example}

\begin{example}[3rd Degree B-Spline Polynomials] 
Here the components of the basis functions are the four polynomial pieces of the B-spline shifted to $[0,1]$:
\[
f(x) = \frac{1}{6}(x^3,-3x^3+3x^2+3x+1, 3x^3 - 6x^2 + 4, (1-x)^3)^T.
\]
The transformation matrix is then 
\[
T_f = \frac{1}{6}\begin{bmatrix} 0&0&0&1\\ 1&3&3&-3\\ 4&0&-6&3\\ 1&-3&3&-1 \end{bmatrix}.
\]
As before we consider the two special cases for the $l_i$ used previously and $M_i= T M_{l_i} T^{-1}$:
\[
M_1 = \begin{bmatrix} \frac{1}{8} & 0 & 0 & 0 \\[4pt] \frac{3}{4} & \frac{1}{2} & \frac{1}{8} & 0 \\[4pt] 
   \frac{1}{8}& \frac{1}{2}&  \frac{3}{4} & \frac{1}{2} \\[4pt] 0 & 0 & \frac{1}{8}& \frac{1}{2} \end{bmatrix}, 
   \quad M_2 = \begin{bmatrix} \frac{1}{2}&\frac{1}{8}&0&0\\[4pt] 
   \frac{1}{2} & \frac{3}{4} & \frac{1}{2}&\frac{1}{8} \\[4pt] 0&\frac{1}{8}&\frac{1}{2}&\frac{3}{4}\\[4pt] 0&0&0&\frac{1}{8} 
   \end{bmatrix}.
   \]
One can see that these two matrices are again stochastic.
\end{example}
\subsection{Polynomial evaluation}
The above considerations lead to the following family of remarkable (because of their derivation via self-similarity) algorithms for evaluating polynomials. An algorithm in base $10$ is described and applied to the evaluation of polynomials on the interval $[0,10]$. 

To this end, let $p:[0,10]\rightarrow \mathbb{R}$ be a polynomial of degree $m\in \N_0$:
\[
p(x)=a_{0}+a_{1}x+...+a_{m}x^{m}\text{.}%
\]
We note that the attractor of the set of fractels%
\[
\{[0,10]\mathbb{\times R}^{m};w_{\frac{1}{10},\frac{10n}{9}},\, n=0,1,2,...,9\}
\]
is $G((1,x,...,x^{m})^{\top})$ restricted to the graph of $p$ for $x\in [0,10].$ In the following, we are omitting the expression ``restricted to the graph of $p$'' when no confusion is expected. It follows that
\begin{align}
p(d_{1}.d_{2}d_{3}d_{4}&....d_{k}00000...)   =(a_{0},a_{1},...,a_{m})\cdot
N_{d_{1}}N_{d_{2}}...N_{d_{k}}
\begin{pmatrix}1\\ 0\\ \vdots\\ 0\end{pmatrix}
\label{rootformula}\\
&  =\text{the first component of }N_{d_{k}}^{\top}N_{d_{k-1}}^{\top
}...N_{d_{1}}^{\top}(a_{0},a_{1},...,a_{m})^{\top}\text{.}\nonumber
\end{align}
The order of approximation can be increased by only applying a single linear transformation to the results of the preceding calculation. Note that the $x$ transformations are $l_{n}(x)=n+\frac{x}{10}=(1-\frac{1}{10})\frac
{n}{1-\frac{1}{10}}+\frac{x}{10}=(1-s)t+sx,$ where $t=n/(1-\frac{1}{10})$.

\begin{example}
\label{example2} Cubic polynomials, in base $10$, on the interval $[0,10]$. We
take $s=\frac{1}{10},$ so that $t=\frac{10n}{9},$ $n=0,1,...,9$.
\[
D_{s}=%
\begin{pmatrix}
1 & 0 & 0 & 0\\
0 & s & 0 & 0\\
0 & 0 & s^{2} & 0\\
0 & 0 & 0 & s^{3}%
\end{pmatrix},\quad
M_{t}=%
\begin{pmatrix}
1 & 0 & 0 & 0\\
t & 1 & 0 & 0\\
t^{2} & 2t & 1 & 0\\
t^{3} & 3t^{2} & 3t & 1
\end{pmatrix}.
\]
Hence, we define
\[
J(n):=(M_{\frac{10n}{9}}D_{\frac{1}{10}}M_{\frac{10n}{9}}^{-1})^{\top}=\allowbreak%
\begin{pmatrix}
1 & n & n^{2} & n^{3}\\
0 & \frac{1}{10} & \frac{1}{5}n & \frac{3}{10}n^{2}\\
0 & 0 & \frac{1}{100} & \frac{3}{100}n\\
0 & 0 & 0 & \frac{1}{1000}%
\end{pmatrix}.
\]
Equation (\ref{rootformula}), applied to a polynomial $p$ of degree less than or equal
to three, implies that
\begin{align*}
p(d_{1}.d_{2}d_{3}d_{4}....d_{k}00000...) =
\begin{pmatrix}
1 & 0 & 0 & 0
\end{pmatrix}
J(d_{k})J\left(  d_{k-1}\right)  ...\allowbreak J(d_{1})%
\begin{pmatrix}
a_{0} & a_{1} & a_{3} & a_{4}%
\end{pmatrix}
^{\top}%
\end{align*}
For instance, for the polynomial $p(x)=1+3x+2x^{2}+x^{3}\allowbreak$ we find,
\[
J(2)J(1)%
\begin{pmatrix}
1 & 3 & 2 & 1
\end{pmatrix}
^{\top}=%
\begin{pmatrix}
\frac{1151}{125} & \frac{303}{2500} & \frac{7}{12\,500} & \frac{1}{1000\,000}%
\end{pmatrix}
^{\top}%
\]
which gives the correct value $p(1.2)=\frac{1151}{125}$. We also have
\[
p(1.23)=%
\begin{pmatrix}
1 & 0 & 0 & 0
\end{pmatrix}
J(3)%
\begin{pmatrix}
\frac{1151}{125} & \frac{303}{2500} & \frac{7}{12\,500} & \frac{1}{1000\,000}%
\end{pmatrix}
^{\top}=\frac{9576\,667}{1000\,000},
\]
thereby illustrating that in order to increase the precision by one digit, we need only
one additional matrix multiplication. This not a surprising, but it is clear
that the work is automatically elegantly organized via the fractel approach.
\end{example}

\begin{remark}
In principle, the IFS method for evaluation of graphs of polynomials has two benefits compared to the direct application of Horner's rule.
\begin{enumerate}
\item The IFS method is self-correcting (for example, in the Chaos Game version, the accuracy of
the approximation of the graph of the polynomial over a fixed interval is
increased at each iteration) whereas Horner's Rule will tend to exaggerate
(rounding) errors made at earlier stages of the calculation. Here we are
assuming that the "eigenvalue 1" issue has been dealt with by replacing the
relevant component by a contractive map.
\item It is not clear how to efficiently increase the number of significant bits produced by Horner's Rule, without
running the same algorithm at higher precision; this is not the case with the
IFS approach, provided one keeps the last computed vector, as in Example
\ref{example2}.
\end{enumerate}
\end{remark}

\subsection{IFSs and approximations for vector valued functions $f(x)$}

Proposition \ref{Lemma1} of Section \ref{sec:1D} can be generalised to vector valued functions
and fractels used in this section. In particular, one can show that  
$$w(x,y) = ((x+\tau)/2, My + (1-M)G(x))$$
is a fractel for $f: [a,b] \rightarrow \R^d$ with $\tau\in[a,b]$ if
$$G(x) = (I-M)^{-1}(f((x+\tau)/2) - M f(x)).$$
One can rewrite this formula as
$$G(x) = f(x) + (I-M)^{-1}(f((x+\tau)/2) - f(x))$$
and by the mean value theorem one gets
$$G(x) = f(x) + \frac{\tau-x}{2} (I-M)^{-1} f^\prime(\xi)$$
for some $\xi\in[a,b]$. A simple approximation of the second term gives
$$G(x) \approx f(x) + \frac{\tau-x}{2} (I-M)^{-1} f^\prime((a+b)/2).$$
This can be used for example with a Taylor approximation of $f(x)$ in the formula
for $G(x)$ to get a numerical approximation. This result can be utilised as before to 
find a local IFS and corresponding approximations for the function $f(x)$.

One may use the resulting vector-valued approximations to derive higher-order approximations
for real functions $u(x)$ by choosing $f(x)$ to be
\begin{itemize}
  \item a vector of shifted values of $u(x)$ like $f(x) = (u(x), u(x+h), u(x+2h))$;
  \item polynomial approximations of $u(x)$ like $f(x) = (p_1(x), p_2(x), u(x))$, for example
       polynomial interpolants of degree one and two;
  \item (local) basis functions with $u(x)$ like $f(x) = (b_1(x), b_2(x), u(x))$, for example, wavelets;
  \item functions and their derivatives, e.g., $f(x) = (u(x), u^\prime(x))$.
\end{itemize}

\section{Conclusion}

Fractals are defined as fixed points of iterated function systems (IFS). Here we consider 
fractal functions, i.e, functions which have a graph that is a fixed point of an IFS. It 
turns out that this property is very general and, in fact, one could say that more or less 
every function which has a domain that is a fixed point of an IFS is a fractal function. However, if the
components of the IFS, the fractels, are chosen carefully, one may obtain powerful classes
of approximations by approximating the fractels. In particular, the choice of the fractels
should be based on (approximate) symmetries or self-referentiability of the function.

The underlying algebra of the fractels which relates to the local, approximate symmetry (semi-) 
groups of a function is discussed in the earlier sections and the application to real and
vector valued functions is considered in the last two sections.

Here, only functions of one variable have been discussed. However, the ideas presented 
can be generalized to functions of multiple variables. This is work in progress.

\section*{Acknowledgement}
The third author was partially supported by DFG grant MA5801/2-1.
The work of the second author partially supported  by the Technische Universit\"at M\"unchen 
-- Institute for Advanced Study, funded by the German Excellence Initiative.

{\vspace{.1in}
\noindent Michael F. Barnsley\\
Mathematical Sciences Institute\\
The Australian National University\\
Canberra, ACT, Australia\\
michael.barnsley@anu.edu.au

\vspace{.1in}\noindent Markus Hegland\\
Mathematical Sciences Institute\\
The Australian National University\\
Canberra, ACT, Australia\\
markus.hegland@anu.edu.au \\
and \\
Institute for Advanced Study\\
Technische Universität München\\
Lichtenbergstrasse 2a\\
D-85748 Garching, Germany.

\vspace{.1in}\noindent Peter Massopust\\
Center of Mathematics, Research Unit M15\\
Technical University of Munich (TUM)\\
Boltzmannstr. 3\\
85747 Garching, Germany\\
massopust@ma.tum.de
}

\end{document}